\documentclass{amsart}

\usepackage{amssymb, amsmath, color}
\usepackage{verbatim}
\usepackage{mathrsfs}
\usepackage{amsmath}
\usepackage{amsfonts}
\usepackage{amssymb}
\usepackage{graphicx}
\usepackage{color}
\usepackage{epstopdf} 
\usepackage{epsfig}
\usepackage{tikz}

\newtheorem{thm}{Theorem}[section]

\newtheorem{prop}[thm]{Proposition}
\newtheorem{wn}[thm]{Corollary}
\newtheorem{obs}[thm]{Observation}

\newtheorem{rem}{Remark}

\newcommand{\sd}{{\rm sd}}
\usepackage{color}

\title{Critical graphs upon multiple edge subdivision}
\author[Magda Dettlaff,\\ Magdalena Lema\'nska, Adriana Roux]{Magda Dettlaff,\\ Magdalena Lema\'nska, Adriana Roux}
\date{}
\begin{document}

\maketitle

 \begin{abstract}
A subset $D$ of $V$ is \emph{dominating} in $G$ if every vertex of $V-D$ has at least one neighbour in $D;$ let $\gamma(G)$ be the minimum cardinality among all dominating sets in $G.$
A graph $G$ is $\gamma$-$q$-{\it  critical} if the smallest subset
of edges whose subdivision necessarily increases 
$\gamma(G)$ has cardinality $q.$ 
 In this paper we consider mainly $\gamma$-$q$-critical trees and give some general properties of $\gamma$-$q$-critical graphs. In particular, we show that if $T$ is a $\gamma$-$q$-critical tree, then $1 \leq q \leq n(T)-1$ and we characterize extremal trees
when $q=n(T)-1.$ Since a subdivision number {of a tree $T$} $\sd(T)$ is always $1,2$ or $3,$ we also characterize  $\gamma$-$2$-critical trees $T$ with $\sd(T)=2$ and $\gamma$-$3$-critical trees $T$ with $\sd(T)=3.$ 

{\AmS \; Subject Classification:} 05C05, 05C69.

\end{abstract}

\section{Introduction}

Let $G =(V ,E)$ be a connected graph of order $n(G)$ and size $m(G)$.

The \emph{neighbourhood} $N_{G}(v)$ of a vertex $v \in V$ is the set of all vertices adjacent to $v$ in $G.$ The \emph{degree} of a vertex $v$ is denoted by $d_{G}(v)=|N_{G}(v)|.$

For a set $X\subseteq V,$ the \emph{open neighbourhood} $N_{G}(X)$ is denoted by $\bigcup_{v\in X}N_{G}(v)$ and the \emph{closed neighbourhood} is denoted by $N_{G}[X]=N_{G}(X)\cup X.$ For a set $S$, let $N_S[x]=N_G[x]\cap S$.

A vertex $v$ is an \emph{end-vertex} {(or a \emph{leaf})} of $G$ if $v$ has exactly one neighbour in $G.$ The set of all end-vertices in $G$ is denoted by $\Omega(G).$

A vertex $v$ is called a \emph{support} if it is adjacent to an end-vertex. If $v$ is adjacent to only one end-vertex, it is called a \emph{weak support}. Otherwise, $v$ is called a \emph{strong support}. The set of all supports in a graph is denoted by $S(G).$

{The {\it distance} between two vertices $u,v$ is the length of the shortest $u-v$ path  in a graph $G$ and is denoted by $d_G(u,v)$. A $u-v$ path of  length $d_G(u,v)$ is called a $u-v$ {\it geodesic}. We say that a set $A\subseteq V$ is a 2-{\it packing} if $d_G(x,y)>2$ for every $x,y\in A$. }

For a graph $G$, the {\it subdivision} of an edge $e=uv$ with a new vertex $w$ (called the \emph{subdivision vertex}) is an operation which leads to a graph~$G_e$ with $V(G_e)=V(G)\cup \{w\}$ and $E(G_e)=(E(G)\setminus\{uv\})\cup \{uw,wv\}$.
Furthermore, the graph obtained from $G$ by subdividing all the edges in the set $F\subseteq E(G)$ is denoted by $G_F$.

 A subset $D$ of $V$ is \emph{dominating} in $G$ if every vertex of $V\setminus D$ has at least one neighbour in $D.$ 
Let $\gamma(G)$ be a minimum cardinality among all dominating sets in $G$ and a dominating set of cardinality $\gamma(G)$ is called a $\gamma$-{\it set} of $G$.  For domination related concepts not defined here, consult \cite{fund}.

The \emph{domination subdivision number}, sd$(G)$, of a graph $G$ is a minimum number of edges which must be subdivided (where each edge can be subdivided at most once) in order to increase the domination number. Since the domination number of the graph $K_2$ does not increase when its only edge is subdivided, we {therefore consider only connected graphs} of order at least $3$. The domination subdivision number was defined by Velammal in 1997 (see \cite{vel}) and since then it is widely studied in graph theory papers.  This parameter was studied for trees in~\cite{fav1} and  ~\cite{myn}. General bounds and properties has been studied for, among others, in~\cite{HHH},  ~\cite{bhv},  ~\cite{fav3}, and  ~\cite{fav2}. 
A graph $G$ is $\gamma$-$q$-{\it  critical} {(shortly $q$-critical)} if the smallest subset
of edges whose subdivision necessarily increases 
$\gamma(G)$ has cardinality $q;$ $q \geq 1.$ {In other words, $G$ is $\gamma$-$q$-critical if $\gamma(G_F)>\gamma(G)$ for every $F\subseteq E$ such that $|F|=q$. From the definition we immediately obtain the following remark.
\begin{rem} Graph $G$ is $q$-critical if subdivision of any $q$ edges changes the domination number of $G$ and there is a set of $q-1$ edges such that a subdivision of these edges does not change the domination number $G$.
\end{rem}}

In this paper we consider {mainly} $\gamma$-$q$-critical trees. In particular, we show that if $T$ is a $\gamma$-$q$-critical tree, then $1 \leq q \leq n(T)-1$ and we characterize extremal trees. The case where $q=1$ was introduced by Rad \cite{rad13} in 2013 where he characterized 1-critical graphs as graphs where every $\gamma$-set is a 2-packing. He also provided a structural characterization of 1-critical trees. We characterize trees for which $q=n(T)-1.$ Since a subdivision number $\sd(T)$ is always $1,2$ or $3,$ we also characterize $\gamma$-$2$-critical trees $T$ with $\sd(T)=2$ and $\gamma$-$3$-critical trees $T$ with $\sd(T)=3.$

\section{Preliminary results}

Note that the domination number of a graph cannot decrease with the subdivision of an edge and can increase with at most one.

\begin{prop}\cite{rad13}\label{prop:increase}
For any edge $e$ in a graph $G$, $\gamma(G)\leq \gamma(G_e)\leq \gamma(G)+1$.
\end{prop}

We begin with some general remarks.

\begin{obs} If there is a $\gamma$-set $D$ in $G$ such that in $V \setminus D$ there exists a vertex having $k$ neighbours in $D,$ then $G$ is not $(k-\ell)$-critical; $1 \leq \ell \leq k-1.$
\end{obs}

\begin{wn} If $G$ is $q$-critical, then for any $\gamma$-set $D$ of $G$ and every $v \in V \textcolor{red}{\setminus}D$ we have $|N_D(v)|\leq q.$
\end{wn}

Since the subdivision of any $k$ edges in the cycle $C_n$ (the path $P_n$) {leads to a graph} isomorphic to $C_{n+k}$ ($P_{n+k}$), we obtain the following observation.

\begin{obs}\label{opath}
If a cycle $C_n$ and a path $P_n$, $n\geq 3$, is $q$-critical, then
\[q=\sd(C_n)=\sd(P_n)=\begin{cases} 1 &\textrm{ if } n\equiv 0 \bmod 3,\\
2 &\textrm{ if } n\equiv 2 \bmod 3,\\
3 &\textrm{ if } n\equiv 1 \bmod 3.
\end{cases}\]
\end{obs}

\begin{obs} \cite{rad13}
If $G$ contains a universal vertex, then $G$ is 1-critical.
\end{obs}

\begin{obs}
If a complete bipartite graph $K_{s,t}$, with $2\leq s\leq t$, is $q$-critical, then
\[q=\begin{cases} t+1 &\textrm{if } s=2,\\
2& \textrm{otherwise.}\\
\end{cases}\]
\end{obs}

\section{$\gamma$-$q$-critical graphs}

We begin this section with some definitions.

   The \emph{corona} $ G \odot H$ is defined as the graph obtained by taking {$n(G)$} copies of a graph $H$ and for each $i\leq n$ adding edges between the $i$th vertex of $G$ and each vertex of the $i$th copy of $H.$
 A \emph{spider} $S_t$ is a graph obtained from the star $K_{1,t}$ for $t \geq 1$ by subdividing each edge of the star.
 
  A \emph{d-wounded spider} $S_{t,t-d}$ is the graph formed by subdividing $t-d\leq t-1$ edges of a star $K_{1,t}, t \geq 1$  ($d$ is a number of edges that we do not subdivide; $d \geq 1$.
 If  $t \geq 2$ and exactly $t-1$ of the edges of a star is subdivided, i.e. $d=1$, then the resulting graph  is called a \emph{slightly wounded spider}.
 
The maximum cardinality of an independent set of $G$ is called the {\it independence number} of $G$ and denoted by $\alpha(G)$. An independent set of cardinality $\alpha(G)$ is called an $\alpha$-{\it set} of $G$.
 
\begin{prop}
  If $T=S_{t,t-k}$ is a $k$-wounded spider, then {$T$ is $q-$critical for} $q=n(T)-k$.
  \end{prop}

\begin{proof}
Note that $\gamma(T)=t-k+1$. Label the vertices of $T$ as follows: label the center vertex of $K_{1,t}$ with $x$ and its leaves with $\{v_1,\dots,v_t\}$. Subdivide the edges $xv_i$ with vertices $u_i$ for $i=k+1, \dots, t$.

Let $A=\{xu_i, u_iv_i|i=k+1,\dots,{t}\}$  be a set of $2(t-k)={n(T)}-k-1$ edges. If the edges in $A$ are subdivided, then the set $D=\{x,y_{k+1},\dots,y_{t}\}$ is a dominating set of cardinality $t-k+1$, where $y_i$ is the {subdivision vertex of} the edge $u_iv_i$.  Therefore, $q\geq {n(T)}-k$.

Let $A'$ be a set of ${n(T)}-k=2(t-k)+1$ edges. Then $A'$ necessarily contains at least one of the edges $xv_i$ for $i\leq k$. The subdivision of the edges in $A'$, producing {$T_{A'}$}, will then increase the number of support vertices of $T$ and hence {$\gamma(T_{A'})\geq |S(T_{A'})|>t-k+1=\gamma(T)$.} It follows that $T$ is $({n(T)}-k)$-critical.
\end{proof}
  
  \begin{wn}
  For each $q$ there exists a $q-$critical tree.
  \end{wn}

Even if the graph is without leaves, we can obtain a similar result. Construct the graph $G_k$  for $k\geq 1$, as follows: take $k$ 4-cycles $H_i\simeq(x_i,y_i,z_i,v_i,x_i)$ for $i=1,\dots,k$. Now indentify vertices $v_i$ with each other to obtain the vertex $v$.

\begin{figure}[h]
 \includegraphics[scale=1]{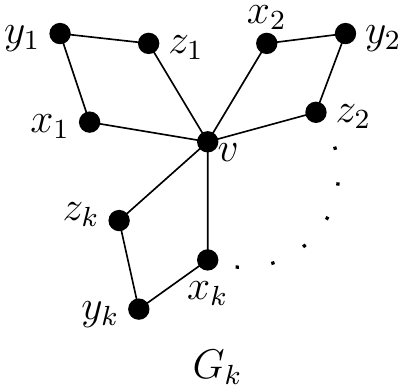}
\caption{The graph $G_k$.}
\label{fig:gk}
\end{figure} 

\begin{prop}
The graph $G_k$ is $q$-critical for $q=n(G_k)-k$ where $k\geq 1$.
\end{prop}
 
\begin{proof}
Note that $\gamma(G_k)=k+1$ and any $\gamma$-set $D$ of $G_k$ contains $v$ and $|D\cap\{x_i,y_i,z_i\}|=1$. Also, $n=n(G_k)=3k+1$. Now let $F=\{vx_i,vz_i\colon i=1,\dots,k\}$ and consider $(G_k)_F$. The set $D'=\{v\}\cup\{y_i\colon i=1,\dots,k\}$ is a $\gamma$-set of $(G_k)_F$ of cardinality $k+1$ and therefore {$q\geq |F|+1=n-k$.}

On the other hand if we subdivide any set $F'$ of $n-k$ edges, then there exists a $j\leq k$ such that $|E(H_j)\cap F'|\geq 3$. This copy becomes a cycle of the length at least 7 and we need at least 3 vertices to dominate it. Therefore $\gamma((G_k)_{F'})> \gamma(G_k)$ and hence $G_k$ is $(n-k)$-critical.
\end{proof}

It is also possible for $q$ to be larger than $n$. Let $A_k$ be the graph obtained from $K_{3,k+5}$, for $k\geq 0$, by adding a leaf to each of the vertices in the partite set $V_1$, where $|V_1|=3$. Let $V_1=\{v_1,v_2,v_3\}$, $V_2=\{u_1,\dots,u_{k+5}\}$ and label the leaves $x_i$ for $i=1,2,3$.

\begin{prop}
The graph $A_k$ is $q$-critical for $q=n(A_k)+k$ where  $k\geq 0$.
\end{prop}

\begin{proof}
Note that $\gamma(A_k)=3$ and $V_1$ is the unique $\gamma$-set of $A_k$. Also, $n=n(A_k)=k+11$. Now let $F=\{v_1u_i,v_2u_i\colon i=1,\dots,k+5 \}$. Then $V_1$ is still a $\gamma$-set of $(A_k)_F$ and therefore $q\geq 2k+11=n+k$.

On the other hand consider any set $F'$ of $n+k$ edges. If there is at least one $u_i$  incident to three edges of $F'$, then $\gamma((A_k)_{F'})> \gamma(A_k)$. Otherwise, every $u_i$ is incident to at most two edges of $F'$ and therefore at least one pendant edge, say $x_1v_1$, belongs to $F'$. If  $v_1u_j\in F'$ for some $j\leq k+5$, then clearly $\gamma((A_k)_{F'})> \gamma(A_k)$. If $v_1u_j\not\in F'$ for all $j\leq k+5$, then $F'$ contains at least $2k+8$ edges of the form $v_iu_j$ for $j=2,3$. It is easy to check that  $\gamma((A_k)_{F'})> \gamma(A_k)$. It now follows that $A_k$ is $(n+k)$-critical.
\end{proof}

\begin{prop}\label{prop:corona}
  Let $G$ be a graph of order $n(G)$ and size $m(G)$. Then $H=G\odot K_1$ is $(m(G)+1+\alpha(G))$-critical.
  \end{prop}

\begin{proof}
Of course, every minimum dominating set of $H=G \odot K_1$ has cardinality $n(G).$ Let $A$ be an $\alpha$-set of $G$. Then $D=V(G) \setminus A$ is a dominating set of $G$. Label the vertices of $G$ as $v_1,v_2,\dots,v_n$, where $v_1,v_2,\ldots v_{\alpha(G)}\in A$ and  label the copies of $K_1$ in $H$ with $u_i$. 

Let $F=E(G)\cup \{v_{i}u_{i},\ i=1,\dots,\alpha(G)\}$ and consider $H_F$. The set $D\cup\{w_{1},\dots,w_{\alpha(G)}\}$, where $w_{i}$ is a subdivision vertex of $v_{i}u_{i}$, is a dominating set of $H_F$ of cardinality $n(G)$ showing that the subdivision of $|F|=m(G)+\alpha(G)$ edges does not increase the domination number of $H$.

Now consider a set $B$ of $m(G)+1+\alpha(G)$ edges of $H$ and let $B'=B\cap\{v_iu_i\colon 1 \leq i \leq n\}$. Then $|B'|\geq \alpha(G)+c$ with $c\geq 1$ and there exist edges $v_ju_j, v_ku_k\in B$ such that $v_jv_k \in E(G)$ where $j,k$ can be chosen in such a way that $v_jv_k \in B$.

 Let us consider $H_B$ and let $D'$ be a $\gamma$-set of $H_B$. Then $D_1=D'\cap\{w_i,u_i\}\neq \varnothing$ for each $v_iu_i\in B'$ where $w_i$ subdivides $v_iu_i$ and $D_2=D'\cap\{v_i,u_i\}\neq\varnothing$ for $v_iu_i\not\in B' $. The set $D_1\cup D_2$ however does not dominate the subdivision vertex of the edge $v_jv_k$ and since $|D_1\cup D_2|\geq n(G)$, we have $|D'|>n=\gamma(H)$. This proves that $H$ is $(m(G)+1+\alpha(G))$-critical.
\end{proof}
 \begin{wn}\label{cor_tree}
  If $G$ is a tree, then  $H=G\odot K_1$ is $(n(G)+\alpha(G))$-critical.
 \end{wn}

  Since $G=K_{1,r}\odot K_1$ has $2r+2$ vertices, $n(K_{1,r})=r+1$ and $\alpha(K_{1,r})=r$, $r\geq 1$  it follows that $K_{1,r} \odot K_1$ is $(n(G)-1)$-critical. 

  \begin{wn}\label{cor1}
  If $T$ is a $q-$critical tree, then $1\leq q\leq n(T)-1$ and these bounds are tight.
 \end{wn}

Rad characterized the case where $q=1$ as follows:

\begin{thm}\cite{rad13}\label{prop:1crit}
A graph $G$ is 1-critical if and only if every $\gamma$-set of $G$ is a 2-packing.
\end{thm}

We show that $K_{1,r}\odot K_1$, $r\geq 1$, are the only graphs which reach the upper bound in Corollary~\ref{cor1}.
  
\begin{thm}
  For a  $q$-critical tree $T$, $q=n(T)-1$ if and only if $T=K_{1,r}\odot K_1$ for some $r\geq 1$ (i.e. $T$ is a slightly wounded spider).
  \end{thm}

\begin{proof}
If $T$ is a slightly wounded spider it follows from {Corollary~\ref{cor_tree}} that $T$ is $(n(T)-1)$-critical.

Now assume that $T$ is $(n(T)-1)$-critical and let $O= V(T)\setminus(\Omega(T)\cup S(T))$. To show that $T$ is a corona graph we show that $T=T' \odot K_1$ for some tree $T',$ (i.e $T$ has only weak supports and leaves and that $O=\varnothing$).

First assume that $T$ has a strong support vertex $x$ with at least two neighbours $x_1,x_2\in \Omega(T)$. Subdividing the edge $xx_i$ increases the number of support vertices and hence also the domination number of $T$, i.e. $\gamma(T_{xx_i})>\gamma(T)$. {Since} any set of $n-2$ edges must contain $xx_1$ or $xx_2${,} $\gamma(T_F)>\gamma(T)$ for any set $F$ of $n-2$ edges. 
It follows that $T$ is not $(n(T)-1)$-critical, a contradiction.

Now assume that $O\neq \varnothing$. Since $T$ is connected there are at least two edges between $S(T)$ and $O$. We consider two cases:

\emph{Case 1.}
If $|O|\leq 2$, then $S(T)$ is a $\gamma$-set of $T$ and there exist two edge-disjoint paths $P_i=(z_i,x_i,y_i)$ where $y_i\in O$ for $i=1,2$. Note that it is possible that $y_1=y_2$.

Let $F_i=\{z_ix_i,x_iy_i\}$ for $i=1,2$ and consider the graph $T_{F_i}$. Suppose that $x_iy_i$ is subdivided by $w_i$. Then $w_i$ is not dominated by a support vertex of $T_{F_i}$ and {$\gamma(T_{F_i})\geq|S(T_{F_i})|+1=|S(T)|+1>\gamma(T)$}. {Since} any set of $n-2$ edges must contain $F_1$ or $F_2$, $\gamma(T_{F})>\gamma(T)$ for any set $F$ of $n-2$ edges, a contradiction.

\emph{Case 2. }
{If $|O|>2$, there exist two non-adjacent vertices  $y_1,y_2\in O$ adjacent to two different vertices in $S(T)$, say $x_1$ and $x_2$, respectively. Let $F_i=\{z_ix_i\colon z_i\in \Omega(T)\}\cup\{y_iv\colon v\in N(y_i)\}$, $i \in \{1,2\}.$  Suppose that the edges $z_ix_i$ and $x_iy_i$ are subdivided by $f_{i,1}$ and $f_{i,2}$, respectively, and that the remaining edges incident to $y_i$ is subdivided by $f_{i,j}$ for $j=3,\dots,d_T(y_i)+1$.}

Now consider $T_{F_i}$. Let $D_{F_i}$ be  a $\gamma$-set of $T_{F_i}$ with a minimum number of subdivision vertices.  Since $\{z_i,f_{i,1}\}\cap D_{F_i}\neq \varnothing$, we consider two subcases:

\emph{Case 2.1. }Let $f_{i,1}\in D_{F_i}$. If $x_i\in D_{F_i}$, then $D=(D_{F_i}-\{f_{i,j}\colon j\geq 1\})\cup \{y_i\}$  is a dominating set of $T$ with $|D|<|D_{F_i}|$. Otherwise, from the choice of $D_{F_i}$ (it has the smallest number of subdivision vertices) $y_i\in D_{F_i}$ and $\{f_{i,j}\colon j\geq 2\}\cap D_{F_i}\neq\varnothing$. Hence, $D=(D_{F_i}-\{y_i,f_{i,1}\})\cup \{x_i\}$  is a dominating set of $T$ with $|D|<|D_{F_i}|$.

\emph{Case 2.2. }If $z_i\in D_{F_i}$, then obviously $f_{i,1}\not \in D_{F_i}$. Assume  $x_i \in D_{F_i}$.  In this case $D=(D_{F_i}-(\{z_i\}\cup \{f_{i,j}\colon j\geq 2\}))\cup \{y_i\}$  is a dominating set of $T$ with $|D|<|D_{F_i}|$. Now let $x_i\not \in D_{F_i}$. If $f_{i,2}\in D_{F_i}$, then from the choice of $D_{F_i}$ we have $(\{y\}\cup \{f_{i,j}\colon j\geq 3\})\cap D_{F_i}=\varnothing$. Thus $D=(D_{F_i}-\{f_{i,2},z_i\})\cup \{x_i\}$  is a dominating set of $T$ with $|D|<|D_{F_i}|$. Finally, if $f_{i,2}\not \in D_{F_i}$, then $y_i\in D_{F_i}$ and $\{f_{i,j}\colon j\geq 3\}\cap D_{F_i}=\varnothing$ (from the choice of $D_{F_i}$). In this case $D=(D_{F_i}-\{y_i,z_i\})\cup \{x_i\}$  is a dominating set of $T$ with $|D|<|D_{F_i}|$. 

Since $F_1\cap F_2=\varnothing$ any set of $n-2$ edges contains $F_1$ or $F_2$. Therefore $\gamma(T_{F})>\gamma(T)$ for any set $F$ of $n-2$ edges, a contradiction.

It follows that $O=\varnothing$ and hence $T=T'\odot K_1$ for a tree $T'$ and $T$ is $q$-critical for $q=n(T)-1$. By Corollary~\ref{cor_tree}, $q=n(T') +\alpha(T')$. Since $n(T)=2n(T')$ it follows that $\alpha(T')=n(T')-1=m(T')$. Thus $T'$ is a star and $T=K_{1,r}\odot K_1$ for $r\geq 1$.
\end{proof}
 
\section{$q$-critical trees with $\sd(T)=q$}

It is known \cite{vel} that for any tree $T$ the subdivision number lies between 1 and 3, therefore we characterize  $\gamma$-$2$-critical trees $T$ with $\sd(T)=2$ and $\gamma$-$3$-critical trees $T$ with $\sd(T)=3.$ 

\subsection{$2$-critical trees}

\begin{thm} \label{t2c} A tree $T$ is $2$-critical if and only if there exists at least one $\gamma$-set with exactly one pair of vertices  $x,y\in D$ such that $d(x,y)\in\{1, 2\}$ and every $\gamma$-set $D$ of $T$ satisfies the following condition: there exists at most one pair of vertices $x,y\in D$ such that $d(x,y)\leq 2$ and each of $x$ and $y$ has at least two neighbours not in $D$.

\end{thm}
\begin{proof}
Suppose that there is no $\gamma$-set $D$ {in $T$} with exactly one pair of vertices  $x,y\in D$ such that $d(x,y)\in\{1, 2\}$. Then every $\gamma$-set in $T$ is a 2-packing or there are more than one pair of vertices at distance less than two in $D.$ In the first case it follows from {Theorem} \ref{prop:1crit} that $T$ is 1-critical.

{Now }suppose that $T$ has a $\gamma$-set $D$ with at least two pairs of vertices $\{x_1,y_1\}$ and $\{x_2,y_2\}$ such that $d(x_i,y_i)\leq 2$, {for $i \in \{1,2\};$} note that it is possible that $\{x_1,y_1\}\cap \{x_2,y_2\}\neq\varnothing$. On the $x_i-y_i$ geodesic, let $v_i$ be the vertex adjacent to $x_i$ (note that it is possible that $v_i=y_i$). If the edge $x_iv_i$ is subdivided with $w_i$, then $x_i$ will dominate $w_i$ and $y_i$ will dominate $v_i$. Hence there exist two edges whose subdivision does not increase the domination number of $T$ and hence $T$ is not 2-critical.

Thus there exists a $\gamma$-set $D$ with exactly one pair of vertices  $x,y\in D$ such that $d(x,y)\in\{1, 2\}$. We may assume that $D\cap \Omega(T)=\varnothing$, otherwise we may exchange a leaf with its support vertex. If this exchange results into the dominating set having two pairs of vertices  $x,y\in D$ such that $d(x,y)\in\{1, 2\}$, then we obtain the case considered in the paragraph above. Suppose 
at least one of $x$ or $y$, say $x$, has at most one neighbours not in $D$.
Since $D$ is a $\gamma$-set of $T,$ $x$ and $y$ has at least one neighbour in $V(T) \setminus D$.
Hence, $x$ has exactly one neighbour $x'\in V(T)\setminus D$. {Since $x$ is not a leaf, $xy\in E(T)$. Subdivide the edges $xx'$ and $xy$ with $w_1$ and $w_2$, respectively, to form $T'$. Then $(D-\{x\})\cup\{w_1\}$ is a dominating set of $T'$ and therefore $T$ is not 2-critical.}

Conversely, assume that every $\gamma$-set of $T$ has the desired property and to the contrary suppose  that $T$ is not 2-critical.  Hence there exists a set of two edges $F=\{e_1=x_1y_1,e_2=x_2y_2\}$ such that $\gamma(T_F)=\gamma(T)$. Let $w_1$ and $w_2$ be the subdivision vertices of $e_1$ and $e_2$, respectively.

Let $D'$ be a $\gamma$-set of $T_F$ with the smallest number of subdivision vertices and consider the following cases.

\emph{Case 1. Edges $e_1$ and $e_2$ are adjacent.} Without the loss of generality assume that $x=x_1=x_2$.
If $x\in D'$, then there is a vertex $z_i\in N[y_i]\cap D'$ for $i=1,2$. From the choice of $D'$, it follows that $w_1,w_2\not\in D'$. Therefore, $D'$ is a $\gamma$-set of $T$ such that $d(x,z_i)\leq 2$ for $i=1,2$, a contradiction.

Now consider the case where $x\not\in D'$. If $w_1,w_2\not\in D'$, then $y_1,y_2\in D'$ and therefore there exists a vertex $x'\in N(x)\setminus\{w_1,w_2\}$ such that $x'\in D'$. Therefore, $D'$ is a $\gamma$-set of $T$ such that $d(x',y_i)\leq 2$ for $i=1,2$, a contradiction.
Thus, at least one of the subdivision vertices belongs to $D'$ and by the choice of $D'$ exactly one, say $w_1$, belongs to $D'$. It is clear that $y_2\in D'$.  If $d_{T_F}(x)>2$, then there exists $x'\in N(x)\setminus\{w_1,w_2\}$. Since $x\not\in D'$ there exists $x''\in {N[x']}\cap D'$ and $D=(D'\setminus\{w_1\})\cup\{x\}$ is a $\gamma$-set of $T$ such that {$d(x'',x)\leq 2$ and $d(x,y_2)=1$}, a contradiction. Now, if $d_{T_F}(x)=2$, then $D=(D'\setminus\{w_1\})\cup\{x\}$ is a $\gamma$-set of $T$ and $y_1$ is the only neighbour of $x$ outside of $D$, a contradiction.

\emph{Case 2. Edges $e_1$ and $e_2$ are not adjacent.} We show that there exists a $\gamma$-set $D$ of $T$ such that for each edge $e_i$ there exists a pair $u_i,v_i\in D$ such that $d(u_i,v_i)\leq 2$.

If $w_1,w_2\not \in D'$, then at least one of $x_1,y_1$, say $x_1$, belongs to $D'$ and at least one of $x_2,y_2$, say $x_2$, belongs to $D'$. Then there exists $z_i\in N[y_i]\setminus \{w_i\}$ such that $z_i\in D'$ for $i=1,2$.   Therefore, $D'$ is a $\gamma$-set of $T$ such that $d(x_i,z_i)\leq 2$ for $i=1,2$, a contradiction. Thus at least one of $w_i$ belongs to $D'.$

Consider the case where $d(\{x_1,y_1\},\{x_2,y_2\})=1$, say $d(x_1,x_2)=1$. If {$w_1,w_2\in D'$}, then by the choice if $D'$, $x_1,x_2, y_1,y_2\not \in D'$. Thus $D=(D'\setminus \{w_1,w_2\})\cup \{x_1,x_2\}$ is a $\gamma$-set of $T$. If $d_T(x_1)>2$, then there exists a vertex {$x''\in D'$} such that $d(x_1,x'')\leq2${. Since $d(x_1,x_2)=1$, $D$ is a $\gamma$-set of $T$ containing two pairs of vertices at distance at most 2, a contradiction. On the other hand, if $d_T(x_1)=2$, then} $y_1$ is the only neighbour of $x_1$ outside of $D$, { also a contradiction.}

Assume now only one of $w_1$ or $w_2$ belongs to $D'$, say $w_1$. Thus by the choice of $D'$  $x_1,y_1\not \in D'$; also, $x_2\not \in D'$, otherwise $D'\setminus \{w_1\})\cup \{y_1\}$ would be a $\gamma$-set of $T_F$ contradicting our choice of $D'$.
Since $D'$ is dominating $y_2\in D'$ and there exist $x'\in N(x_2)\setminus\{x_1,w_2\}$ such that $x'\in D'$. Then $D=(D'\setminus \{w_1\})\cup \{x_1\}$is a $\gamma$-set of $T$ and $d(x_1,x')=d(x',y_2)=2$, a contradiction.

Now consider the case where $d(\{x_1,y_1\},\{x_2,y_2\})>1$ and at least one of $w_1,w_2$ is in $D'$, say $w_1$. Then $x_1,y_1\not \in D'$ and since $T$ is connected at least one of $x_1$ or $y_1$ has degree greater than 1, say $x_1$. Therefore there exists a vertex $x''$ such that $d(x_1,x'')=2$ and $x''\in D'$. Similarly, if $w_2\in D'$, there {exists} $y''\in D'$ such that $d(x_2,y'')=2$. Otherwise, if $w_2\not\in D'$ then without the loss of generality $x_2\in D'$ and there {exists} a vertex $z\in N[y_2]\setminus\{w_2\}$ such that $z\in D'$. In both cases $D=(D'\setminus\{w_i\})\cup\{x_i\}$ is a $\gamma$-set of $T$ containing two pairs of vertices at distance at most 2, a contradiction. Hence $T$ is 2-critical. 
\end{proof}

 Let $\mathcal{N}(G)$ consists of those vertices which are not contained in any $\gamma (G)$-set. Benecke and Mynhardt~\cite{myn} characterized all trees with domination subdivision number equal to~1 {as follows:}
\begin{thm}\cite{myn}\label{tt1}
For a tree T of order $n\geq 3$, $\sd(T)=1$ if and only if $T$ has 
\begin{itemize}
\item[$i)$] a leaf $u\in \mathcal{N}(T)$ or
\item[$ii)$] an edge $xy$ with $x,y\in \mathcal{N}(T)$.
\end{itemize}
\end{thm} 

\begin{thm}
The only 2-critical tree $T$ with $\sd(T)=2$ is $T=P_{3k+2}$  for $k\geq 1$.
\end{thm}

\begin{proof}
Let $T$ be a $2$-critical tree such that $\sd(T)=2$. Hence by Theorem~\ref{tt1}, $T$ has no strong support vertices. Let $D$ be a  $\gamma(T)$-set and $L=\Omega (T)\cap D$.  Then $D'=(D\setminus L)\cup N(L)$ is also a $\gamma$-set of $T$. If $|L|> 1$, then $D'$ would contain more than one pair of vertices  at distance at most 2, contradicting Theorem~\ref{t2c}. Thus, $|L|\leq 1$  for any  $\gamma(T)$-set $D$. If there exists $\gamma(T)$-set $D$ such that $|L|=1$, then $D$ is a 2-packing. Otherwise, either we could find more than one pair of vertices  at distance at most 2 in $D'$ or $x\in L$ would be at distance at most 2 from another vertex in $D$ and $x$ would have only one neighbour in $V(T)-D$, contradicting the assumption that $T$ is 2-critical. It is easy to observe that in this case if $D$ contains one leaf $x$ of $T,$ then any $\gamma$-set of $T-N[x]$ is 2-packing, so $D$ is a unique $\gamma$-set of $T$ such that $L=\{x\}$ and $D$ is 2-packing. 

Assume that there exists a vertex $v$ such that $d_T(v)\geq 3$. Root $T$ at $v$ and label the subtree rooted at $v_i$, where $v_i\in N(v)$, with  $T_i$ for $1\leq i\leq d_T(v)$.

For each $1\leq i\leq d_T(v)$, let $u_i\in \Omega(T)\cap V(T_i)$ and let $s_i\in N(u_i)$. Since $\sd(T)=2$ it follows from Theorem~\ref{tt1} that there exists a $\gamma(T)$-set $D_i$ such that $u_i\in D_i$.  Obviously, $D'_i=(D_i- \{u_i\})\cup \{s_i\}$ is a $\gamma(T)$-set and $D'_i-\{s_i\}$ is a unique 2-packing such that $S(T)- \{s_i\}\subseteq D'_i$. This implies that $D_i\cap V(T_j)=D'_i\cap V(T_j)=D'_k\cap V(T_j)=D_k\cap V(T_j)$ for every $i\neq j\neq k\in\{1,\dots, d_T(v)\}$ (for example $D'_1\cap V(T_2)=D'_3\cap V(T_2)$ and  $D'_1\cap V(T_3)=D'_2\cap V(T_3)$ and $D'_2\cap V(T_1)=D'_3\cap V(T_1)$). It follows that either $v\in D_i$ for each $1\leq i\leq d_T(v)$ or $v\not\in D_i$ for each $1\leq i\leq d_T(v)$. 

If $v\in D_i$ for each $1\leq i\leq d_T(v)$, then $D^*=(D_1- V(T_2))\cup (V(T_2)\cap D_2)$ is a $\gamma(T)$-set with more than one leaf, a contradiction. Otherwise, $v$ is dominated by $v_j$ for $1\leq j\leq d_T(v)$ and $\{v_i\colon i\neq j\}\cap D=\varnothing$. Then $D^*=(D_j- V(T_\ell))\cup (V(T_\ell)\cap D_\ell)$ where $\ell\neq j$  is a $\gamma(T)$-set with more than one leaf, a contradiction. 

It follows that $T$ is a path and from Observation~\ref{opath}, $T=P_{3k+2}$ for $k\geq 1$.
\end{proof}

\subsection{$3$-critical trees}

The following constructive characterization of the family $\mathcal{F}$ of labeled trees $T$ with $\sd (T)=3$ was given by Aram, Sheikholeslami and Favaron~\cite{fav1}. 

 Let $\mathcal{F}$ be the family of labelled trees such that $\mathcal{F}$:
 \begin{itemize}
\item contains $P_4$ where the two leaves have status $A$ and the two support vertices have status $B$; and
\item is closed under the two operations  $\mathcal{T}_1$ and  $\mathcal{T}_2$, which extend the tree $T$ by attaching a path to a vertex $v \in V (T )$.

\begin{enumerate}
\item [Operation  $\mathcal{T}_1$.] Assume $sta(v) = A.$ Then add a path $(x,y,z)$ and the edge $vx.$ Let $sta(x)=sta(y) = B$ and $sta(z) = A.$
\item [Operation  $\mathcal{T}_2$.] Assume $sta(v) =B.$ Then add a path $(x,y)$ and the edge $vx.$ Let $sta(x)=B$ and $sta(y)=A.$
\end{enumerate}
\end{itemize}

If $T\in \mathcal{F}$, we let $A(T)$ and $B(T)$ be the set of vertices of statuses $A$ and $B$, respectively, in $T$. It was shown in \cite{fav1} that $A(T)$  is a $\gamma(T)$-set and contains all leaves of $T$.

\begin{thm}\textnormal{\cite{fav1}}
For a tree T of order $n\geq 3$, 
\[
\sd(T)=3 \textrm{ if and only if } T\in \mathcal{F}.
\]
\label{tw2}
\end{thm}

We use this result to show that that paths of order $3k+1$, {for $k\geq 1$,} are the only 3-critical trees with $\sd(T)=3.$

\begin{thm}
The only 3-critical tree $T$ with $\sd(T)=3$ is $T=P_{3k+1}$ for $k\geq 1$.
\end{thm}

\begin{proof}
Let $T$ be a tree with $\sd(T)=3$. By Theorem~\ref{tw2}, $T\in \mathcal{F}$ and there exist a $\gamma(T)$-set $D$ containing all the leaves. Let $\Omega(T)=\{v_1,\dots,v_{\ell}\}$. Since $T\in \mathcal{F}$ each $v_i$ has exactly one neighbour, say $u_i$. Now, let $F=\{u_iv_i\colon i=1,\dots,\ell\}$ {and} consider the graph $T_F$ where the subdivision vertices are denoted by $w_i$, respectively. Then $(D \textcolor{red}{\setminus} \Omega(T))\cup\{w_i\colon i=1,\dots,\ell\}$ is a $\gamma$-set of $T_F$. It {therefore} follows that if $T$ is $q$-critical, then $q> |\Omega(T)|$. Hence if $T$ is 3-critical, $|\Omega(T)|=2$ and $T$ is a path. From the Observation~\ref{opath} $T=P_{3k+1}$ for $k\geq1$.
\end{proof}

\end{document}